\documentclass[11pt]{amsart}

\usepackage{AFBoix2015}

\begin{document}

\author[I. Blanco-Chac\'{o}n]{Iv\'{a}n Blanco-Chac\'{o}n}
\thanks{I. B-C. is supported by Science Foundation Ireland Grant 13/IA/1914 and by MINECO grant MTM2016-79400-P}
\address{School of Mathematical Sciences, University College Dublin, Belfield, Dublin 4, Ireland.}
\email{ivan.blanco-chacon@ucd.ie}

\author[A.\,F.\,Boix]{Alberto F.\,Boix}

\address{Department of Mathematics, Ben-Gurion University of the Negev, P.O.B. 653 Beer-Sheva 8410501, ISRAEL.}
\email{fernanal@post.bgu.ac.il}
\thanks{A. F.B. is supported by Israel Science Foundation (grant No. 844/14) and Spanish Ministerio de Econom\'ia y Competitividad MTM2016-7881-P}

\author[S. Fordham]{Stiof\'ain Fordham}
\address{School of Mathematical Sciences, University College Dublin, Belfield, Dublin 4, Ireland.}
\thanks{S. F. and E.S.Y. are supported by Science Foundation Ireland Grant 13/IA/1914}
\email{stiofain.fordham@ucdconnect.ie}

\author[E. S. Yilmaz]{Emrah Sercan Yilmaz}
\address{School of Mathematical Sciences, University College Dublin, Belfield, Dublin 4, Ireland.}
\email{emrahsercanyilmaz@gmail.com}



\title{Differential operators and hyperelliptic curves over finite fields}

\begin{abstract}
Boix, De Stefani and Vanzo have characterized ordinary/supersingular elliptic curves over $\mathbb{F}_p$ in terms of the level of the defining cubic homogenous polynomial. We extend their study to arbitrary genus, in particular we prove that every ordinary  hyperelliptic curve $\mathcal{C}$ of genus $g\geq 2$ has level $2$. We provide a good number of examples and raise a conjecture.
\end{abstract}

\maketitle

\section{Introduction}

Let $k$ be any perfect field and $R=k[x_1,...,x_d]$ its polynomial ring in $d$ variables. In this case it is known \cite[IV, Th\'eor\`eme 16.11.2]{EGAIV} that the ring $\DD$ of $k$--linear differential operators on $R$ is the $R$-algebra (which we take here as a definition)
\[
\DD:=R \left\langle D_{x_i,t} \mid i=1,\ldots,d \mbox{ and } t\geq 1 \right\rangle\subseteq\mathrm{End}_{k}(R), 
\]
generated by the operators $D_{x_i,t}$, defined as
\[
D_{x_i,t}(x_j^s)=\begin{cases}
\binom{s}{t}x_i^{s-t},\text{ if }i=j\text{ and }s\geq t,\\
0,\text{ otherwise }.\end{cases}
\]
For a non-zero $f\in R$, the natural action of $\DD$ on $R$ extends to $R_f$ in such a way that $R_f=\DD\frac{1}{f^m}$, for some $m\geq 1$. Whilst there are examples of $m>1$ in characteristic $0$ (e.g. \cite[Example 23.13]{Twentyfourhours}), it is $m=1$ in positive characteristic (\cite[Theorem 3.7 and Corollary 3.8]{AlvarezBlickleLyubeznik2005}). This is shown by proving the existence of a differential operator $\delta\in\DD$ such that $\delta(1/f)=1/f^p$, i.e., $\delta$ acts as Frobenius on $1/f$. We will suppose that $k=\mathbb{F}_p$ and fix an algebraic closure $\overline{k}$ of $k$ from now on.

For an integer $e\geq 0$, let $R^{p^e}\subseteq R$ be the subring of all the $p^e$ powers of all the elements of $R$ and
set $\DD^{(e)}:=\End_{R^{p^e}} (R)$, the ring of $R^{p^e}$-linear ring-endomorphism of $R$. Since $R$ is a finitely generated $R^p$-module, by \cite[1.4.8 and 1.4.9]{Yekutiely1992}, it is 
\[
\DD=\bigcup_{e\geq 0}\DD^{(e)}.
\label{filter0}
\]
Therefore, for $\delta\in\DD$, there exists $e\geq 0$ such that $\delta\in \DD^{(e)}$ but $\delta\not\in \DD^{(e')}$ for any $e'<e$. Such number $e$ is called the level of $f$.

The level of a polynomial has been studied in \cite{AlvarezBlickleLyubeznik2005} and \cite{BoixDeStefaniVanzo2015}. In \cite{BoixDeStefaniVanzo2015}, an algorithm is given to compute the level and a good number of examples are exhibited. Moreover, if $f$ is a cubic smooth homogeneous polynomial defining an elliptic curve $\mathcal{C}=V(f)=\{(x:y:z)\in \mathbb{P}^2_k: f(x,y,z)=0\}$, the level of $f$ can be used to characterise the supersingularity of $\mathcal{C}$ in the following way:

\begin{teo}(\cite[Theorem 1.1]{BoixDeStefaniVanzo2015})
Let $f\in R$ be a cubic homogeneous polynomial such that $\mathcal{C}=V(f)$ is an elliptic curve over $k$. Denote by $e$ the level of $f$. Then
\begin{enumerate}[(i)]
\item $\mathcal{C}$ is ordinary if and only if $e=1$.
\item $\mathcal{C}$ is supersingular if and only if $e=2$.
\end{enumerate}
\label{intro1}
\end{teo}

The goal of the present work is to extend the results of \cite{BoixDeStefaniVanzo2015} to hyperelliptic curves of genus $g\geq 2$ defined over $k$. Such a curve $\mathcal{C}$ is birationally equivalent to the vanishing locus $V(f):=\{(x:y:z)\in \mathbb{P}^2_k: f(x,y,z)=0\}$, where $f$ is a homogeneous polynomial of degree $2g+1$ defined over $k$. We will write $\mathcal{C}\cong V(f)$. Notice that $V(f)$ is singular at the infinity point.

Denote by $\mathrm{Jac}(\mathcal{C})$ its Jacobian. It is well known that for
an integer $n>0$ then
\begin{align*}
    \mathrm{Jac}(\mathcal C)[n](\overline k)&=(\mathbb{Z}/n
    \mathbb{Z})^{2g}\quad \text{if $\mathrm{char}\, k  \not|\ n$,}\\
    \mathrm{Jac}(\mathcal C)[p^m](\overline
    k)&=(\mathbb{Z}/p^m\mathbb{Z})^i\quad \text{if $p=\mathrm{char}\, k,\
    m>0$},
\end{align*}
where $i$ can take every value in the range $0\le i \le g$, and is called the
$p$-rank of $\mathcal C$. 

For the convenience of the reader, we recall here the following standard terminology, which will be used in this work

\begin{df}The curve $\mathcal{C}$ is said to be ordinary if its
$p$-rank is maximal. The curve $\mathcal{C}$ is said to be supersingular (resp.\
superspecial) if $\mathrm{Jac}(\mathcal{C})$ is isogenous (resp.\ isomorphic)
over $\overline{k}$ to the product of $g$ supersingular elliptic curves.
\end{df}
Our generalisation of Theorem \ref{intro1} is as follows:

\begin{teo}Let $f\in R$ be a homogeneous polynomial of degree $2g+1$ such that $\mathcal{C}\cong V(f)$ defines a hyperelliptic curve over $\overline{k}$ of genus $g$ and denote by $e$ the level of $f$. Assume $p>2g^2-1$. Then
\begin{itemize}
        \item[(i)] $e=2$ if $\mathcal C$ is ordinary,
        \item[(ii)] $e>2$ if $\mathcal C$ is supersingular but not
          superspecial.
    \end{itemize}
\label{intro2}
\end{teo}
The remaining cases are more complicated and require further study, in particular it would be very interesting to find a relation (if such a relation exists), between the level, the $p$-rank, and the $\alpha_p$-number. We give some examples showing that, beyond the ordinary and supersingular cases, the level is not enough to determine the $p$-rank, namely, we exhibit: a) hyperelliptic curves of genus $2$ having $p$-rank $1$ and level $e=2$, b) hyperelliptic
curves of genus $2$ having $p$-rank $1$ and level $3$. Hence the level cannot be used to characterise the $p$-rank, being then $e=2$ a strict necessary condition for $\mathcal{C}$ to be ordinary. 

The superspecial case also remains open, although we exhibit c) two infinite families of superspecial hyperelliptic curves of arbitrary genus $g\geq 2$ such that, for a fixed genus $g$, there exist an infinite family of primes for which the curve has level $e>2$. Due to these explicit results plus computational evidence, we make the following

\begin{con}Let $f\in R$ be a homogeneous polynomial of degree $2g+1$ such that $\mathcal{C}\cong V(f)$ is a hyperelliptic curve over $\overline{k}$ of genus $g\geq 2$. Denote by $e$ the level of $f$. If $\mathcal{C}$ is superspecial, then $e>2$.
\end{con}

Our article is structured as follows: in Section \ref{preliminaries section} we review some definitions and properties of the level, following \cite{BoixDeStefaniVanzo2015}, as well as useful characterisations of the $p$-rank in terms of the Cartier-Manin matrix, following \cite{Yui1978}. We prove that the level is invariant under homogeneous change of coordinates, however, as we show via an example, the level is not invariant under birational transformation, and we will define the level of a hyperelliptic curve as the level of the polynomial defining its imaginary model . In Section \ref{results on the level} first, we prove that on the contrary to the elliptic case, the level of a hyperelliptic curve of genus $g\geq 2$ must be at least $2$. Then, we prove some technical results (Propositions 3.2 and 3.3), and use them to prove our main result, labelled above as Theorem \ref{intro2} and labelled as Theorem \ref{main result in the ordinary case} in Section \ref{results on the level}. Theorem \ref{main result in the supersingular case} relates the maximal non-vanishing power of the Cartier-Manin matrix with the level, which allows us to conclude the following:
\begin{cor}Let $\mathcal{C}$ be a hyperelliptic curve of genus $g\geq 2$ defined by a homogeneous polynomial $f$ of degree $2g+1$. Denote by $e$ the level of $f$. Assume $\mathcal{C}$ is supersingular but not superspecial. Then, $e>2$.
\end{cor}
Finally, in Section \ref{some examples} we provide several numerical examples of hyperelliptic curves of genus 2 with $p$-rank 1 and different level (Subsection \ref{non ordinary examples}) and exhibit two infinite families of supersingular hyperelliptic curves of genus $g\geq 2$ with level bigger than $2$. On the grounds of this evidence, we close our work by posing a conjecture:

\begin{con}Let $\mathcal{C}$ be a superspecial hyperelliptic curve of genus $g$ over $k$. Then, the level is strictly bigger than $2$.
\end{con}

The present work extends the research already carried out by the authors for genus $2$ and presented in the 13th International Conference on Finite Fields and Applications, held in Gaeta, in June 2017. We take the occasion to thank the organisers for having allowed us to present our preliminary version.

Likewise, the authors are indebted to Gary McGuire for regular discussions on this topic as well as for carefully reading our manuscript and useful comments which have improved our work.

The following notation will be used throughout the article: for $\mathbf{\alpha}=(a_1,\ldots ,a_d)\in\N^d$ we shall denote $\mathbf{x}^{\mathbf{\alpha}}:=x_1^{a_1}\cdots x_d^{a_d}.$ and set $\norm{\mathbf{x}^{\mathbf{\alpha}}}:=\max\{a_1,\ldots ,a_d\}$, and sometimes, even $\norm{\mathbf{\alpha}}$ instead of $\norm{\mathbf{x}^{\mathbf{\alpha}}}$. For any polynomial $g\in k [x_1,\ldots ,x_d]$, we define
\[
\norm{g}:=\max_{\mathbf{x}^{\mathbf{\alpha}}\in\supp (g)} \norm{\mathbf{x}^{\mathbf{\alpha}}},
\]
where if $g=\sum_{\mathbf{\alpha}\in\N^d} g_{\mathbf{\alpha}} \mathbf{x}^{\mathbf{\alpha}}$, $\supp (g):=\left\{x^{\mathbf{\alpha}}\in R\mid\ g_{\mathbf{\alpha}}\neq 0\right\}$. Finally, the total degree of $g$ is
\[
\\mathrm{deg} (g):=\max_{\mathbf{x}^{\mathbf{\alpha}}\in\supp (g)} \mathrm{deg} (\mathbf{x}^{\mathbf{\alpha}}),\text{ where }\mathrm{deg}(\mathbf{x}^{\mathbf{\alpha}}):=
a_1+\ldots +a_d.
\]

\section{Preliminaries}\label{preliminaries section}

\subsection{Reminders about differential operators in prime characteristic}

First, let us denote $\DD^n:=R \left\langle D_{x_i,t} \mid i=1,\ldots,d \mbox{ and } 1\leq t\leq n \right\rangle$ so that one has
\begin{equation}
\DD=\bigcup_{e\geq 0}\DD^n.
\label{filter1}
\end{equation}
From \cite[1.4.8 and 1.4.9]{Yekutiely1992}, it is $\DD^{p^e-1}=\DD^{(e)}$ for each $e\geq 0$, and hence, by \ref{filter1}, it is also
\begin{equation}
\DD=\bigcup_{e\geq 0}\DD^{p^e-1},
\label{filter2}
\end{equation}
making $\DD$ a filtered ring.

\subsection{The ideal of $p^e$-th roots} For any ideal $J \subseteq R$, and for any integer $e\geq 1$, denote by $J^{[p^e]}$ the ideal generated (over $R$) by all the $p^e$-th powers of elements in $J$, or equivalently, by the $p^e$-th powers of any set of generators of $J$.

\begin{df}\cite[page 465]{AlvarezBlickleLyubeznik2005} and \cite[Definition 2.2]{BlickleMustataSmith2008}
Given $g\in R$, for any integer $e\geq 1$, we set $I_e(g)$ to be the smallest ideal $J\subseteq R$ such that $g\in J^{[p^e]}$.
\end{df}
\begin{rk}Since $R$ is a free $R^{p^e}$-module for each $e\in \N$, with basis given by the monomials $\{\mathbf{x}^{\mathbf{\alpha}} \mid \norm{\alpha} \leq p^e-1\}$, for $g\in R$ such that
\[
g=\sum_{0\leq\norm{\mathbf{\alpha}}\leq p^{e}-1}g_{\mathbf{\alpha}}^{p^e} \mathbf{x}^{\alpha},
\]
then $I_e (g)$ is the ideal of $R$ generated by elements $g_{\mathbf{\alpha}}$ \cite[Proposition 2.5]{BlickleMustataSmith2008}.
\label{remarkI}
\end{rk}

\begin{rk} \label{I_e homog} Notice that, if $g$ is a homogeneous polynomial, then, for all $e \in \N$, $I_e(g)$ is a homogeneous ideal. Indeed, if we write $g=\sum_{0\leq\norm{\mathbf{\alpha}}\leq p^{e}-1}g_{\mathbf{\alpha}}^{p^e} \mathbf{x}^{\alpha}$, then we can assume without loss of generality that every $g_\alpha^{p^e} \mathbf{x}^{\alpha}$ has degree equal to $\mathrm{deg}(g)$. Moreover, $g_\alpha$ is homogeneous of degree
\[
\mathrm{deg}(g_\alpha) = \frac{\mathrm{deg}(g)-\mathrm{deg}(\mathbf{x}^{\alpha})}{p^e}.
\]
\label{remarkdeg}
\end{rk}

The relation between ideals of $p^e$--th roots and differential operators is as follows: 

\begin{lm}\cite[Lemma 3.1]{AlvarezBlickleLyubeznik2005}
For any integer $e\geq 0$, it is $I_e (g)^{[p^e]}=\DD^{(e)}\cdot g$, where
$\DD^{(e)}\cdot g:=\{\delta (g):\ \delta\in\DD^{(e)}\}.$
\label{some piece of Frobenius descent}
\end{lm}

\subsection{The level of a polynomial}As pointed out in the introduction, the level of a non-zero polynomial $f\in R$ is defined as the minimal integer $e\geq 1$ such that there exist $\delta\in \DD^{p^e-1}$ with $\delta(1/f)=1/f^p$. Here we make this idea more precise and recall another characterization, following \cite[Definition 2.6]{BoixDeStefaniVanzo2015}. First, we observe that as an immediate consequence of \cite[Lemma 3.4]{AlvarezBlickleLyubeznik2005}, there is a decreasing chain of ideals
\begin{equation}\label{cadena incordio}
R=I_0 (f^{p^0-1})\supseteq I_1 (f^{p-1})\supseteq I_2(f^{p^2-1})\supseteq I_3 (f^{p^3-1})\supseteq\ldots
\end{equation}
The level can be characterised as the minimal step where this chain stabilises. More precisely:

\begin{teo}\cite[Proposition 3.5, and Theorem 3.7]{AlvarezBlickleLyubeznik2005}
Define
\[
e:=\inf\left\{s\geq 1\mid\ I_{s-1}\left(f^{p^{s-1}-1}\right)=I_s\left(f^{p^s-1}\right)\right\}.
\]
Then, the following assertions hold.
\begin{enumerate}[(i)]

\item The chain of ideals \eqref{cadena incordio} stabilizes rigidly, that is $e < \infty$ and $I_{e-1}\left(f^{p^{e-1}-1}\right)=I_{e+s} \left(f^{p^{e+s}-1}\right)$ for any $s\geq 0$.

\item One has
\[
e=\min\left\{s\geq 1\mid\ f^{p^s-p}\in I_s\left(f^{p^s-1}\right)^{[p^s]}\right\}.
\]

\item There exists $\delta\in\DD^{(e)}$ such that $\delta(f^{p^e-1}) = f^{p^e-p}$, or equivalently such that $\delta (1/f)=1/f^p$.

\item There is no $\delta '\in\DD^{(e')}$, with $e'<e$, such that $\delta ' (1/f)=1/f^p$.
\end{enumerate}
\label{first step of the algorithm}
\end{teo}

\begin{df}\label{the level for any polynomial}(\cite[Definition 2.6]{BoixDeStefaniVanzo2015})
For a non-zero polynomial $f \in R$, we call the integer $e$ defined in Theorem \ref{first step of the algorithm} the level of $f$. Also, we will say that $\delta \in \DD^{(e)}$ such that $\delta(f^{p^e-1}) = f^{p^e-p}$, or equivalently such that $\delta(1/f) = 1/f^p$, is a differential operator associated with $f$.
\end{df}

\subsection{The level of a hypersurface}Next, we prove that the level of a non-zero homogeneous polynomial $f\in R$ is an invariant under change of homogeneous coordinates, hence depending just on the projective hypersurface defined by $f$.

Let $k$ be (for the moment) any perfect field of prime characteristic $p$, let $R=k[x_1,\ldots ,x_d],$ standardly $\Z$--graded (meaning $\mathrm{deg} (x_i)=1$ for any $1\leq i\leq d$). Denote $G:=\operatorname{GL}_d (k)$ and observe that $R$ has a right action of $G$ defined by $(f|A)(x_1,...,x_d):=f(y_1,...,y_d),$ where
\[
\begin{pmatrix} y_1\\ \vdots\\ y_d\end{pmatrix}=A\cdot\begin{pmatrix} x_1\\ \vdots\\ x_d\end{pmatrix},
\]
for $A\in G$.
Observe as well that a matrix $A\in G$ induces an isomorphism $\phi_A$ of graded $k$--algebras $\xymatrix@1{R\ar[r]^-{\phi_A}& R}$ defined by $\phi_A(f)=f|A$. This is clear, since $A$ being invertible, both sets $\{\mathbf{x}^{\alpha}:\ \alpha\in\mathbb{Z}_{\ge0}^d\}$ and
$\{\mathbf{y}^{\beta}:\ \beta\in\mathbb{Z}_{\ge0}^d\}$ are $k$--basis of $R$ and for matrices $A,B\in G$ and $f\in R$, it is easy to check that $(f|A)|B=f|AB$.

\begin{df}
Given homogeneous $f,g\in R$, we say that $f$ and $g$ are $G$--equivalent if there is $A\in G$ such that $\phi_A(f)=g$.
\end{df}

\begin{lm}\label{homogeneous change of coordinates give p basis}
Notations as before, let $y_1,\ldots ,y_d\in R$ be homogeneous elements of degree $1$ such that
\[
\begin{pmatrix} y_1\\ \vdots\\ y_d\end{pmatrix}=A\cdot\begin{pmatrix} x_1\\ \vdots\\ x_d\end{pmatrix},
\]
for some $A\in G$ Then, for any $e\geq 1$ the set
\[
\mathcal{B}:=\{\mathbf{y}^{\alpha}:=y_1^{a_1}\cdots y_d^{a_d}:\ \alpha =(a_1,\ldots ,a_d)\in\mathbb{Z}_{\ge0}^d,\ 0\leq a_i\leq p^e-1\text{ for any }1\leq i\leq d\}
\]
is a basis of $R$ as $R^{p^e}$--module.
\end{lm}

\begin{proof}
Since $G$ is invertible, both $\{\mathbf{x}^{\alpha}:\ \alpha\in\mathbb{Z}_{\ge0}^d\}$ and
$\{\mathbf{y}^{\beta}:\ \beta\in\mathbb{Z}_{\ge0}^d\}$ are $k$--basis of $R$, hence it is enough to check that
\[
\mathcal{B}:=\{\mathbf{x}^{\alpha}:=x_1^{a_1}\cdots x_d^{a_d}:\ \alpha =(a_1,\ldots ,a_d)\in\mathbb{Z}_{\ge0}^d,\ 0\leq a_i\leq p^e-1\text{ for any }1\leq i\leq d\}
\]
is a basis of $R$ as $R^{p^e}$--module.

Indeed, let $g\in R,$ and write
\[
g=\sum_{\beta\in\mathbb{Z}_{\ge0}^d}g_{\beta}\mathbf{x}^{\beta},
\]
where $g_{\beta}\in k$ and $g_{\beta}=0$ up to a finite number of
terms; in this way, given $\beta=(b_1,\ldots ,b_d)\in\mathbb{Z}_{\ge0}^d$ so that
$\mathbf{x}^{\beta}\in\supp (g),$ after doing the Euclidean division
we can write in a unique way $b_i=q_i p^e+a_i,$ where both $q_i$
and $a_i$ are non--negative integers, and $0\leq a_i\leq p^e-1.$ Therefore, one
has that
\[
g=\sum_{\alpha\in\mathcal{B}}g_{\alpha}^{p^e}\mathbf{x}^{\alpha},
\]
for some $g_{\alpha}\in R$.
\end{proof}

We can now prove that the level is $G$--invariant, what is immediate after the following result.

\begin{teo}
Let $f,g\in R$ be homogeneous $G$--equivalent polynomials via $A\in G$. Then, for any $e\geq 0$, it is
$\phi_A (I_e (f))=I_e(g)$ and $\phi_A^{-1} (I_e (g))=I_e(f),$ where, given any ideal $J\subseteq R$ and a ring endomorphism $\xymatrix@1{R\ar[r]^-{\varphi}& R},$ $\varphi (J)$ denotes the ideal generated by the images of elements of $J$ under $\varphi.$
\label{the level is a geometric invariant}
\end{teo}

\begin{proof}
Setting
\[
\begin{pmatrix} y_1\\ \vdots\\ y_d\end{pmatrix}=A\cdot\begin{pmatrix} x_1\\ \vdots\\ x_d\end{pmatrix},
\]
and applying Lemma \ref{homogeneous change of coordinates give p basis} we see that for any $e\geq 1,$ the set
\[
\{\mathbf{y}^{\alpha}:=y_1^{a_1}\cdots y_d^{a_d}:\ \alpha =(a_1,\ldots ,a_d)\in\mathbb{Z}_{\ge0}^d,\ 0\leq a_i\leq p^e-1\text{ for any }1\leq i\leq d\}
\]
is a basis of $R$ as $R^{p^e}$--module.

Now, for $e\geq 0,$ and write
\[
f=\sum_{0\leq\norm{\alpha}\leq p^e-1}f_{\alpha}^{p^e}\mathbf{x}^{\alpha},
\]
for $f_{\alpha}\in R$. Then
\[
g=\phi_A(f)=\sum_{0\leq\norm{\alpha}\leq p^e-1}\phi_A(f_{\alpha})^{p^e}\mathbf{y}^{\alpha},
\]
which shows that $\phi_A (I_e (f))\subseteq I_e (g)$. Equality holds because 
$\phi_A$ is an isomorphism.
\end{proof}

Motivated by Theorem \ref{the level is a geometric invariant}, we introduce the following

\begin{df}\label{the level of a homogeneous hypersurface}
Let $H\subseteq\mathbb{P}_{k}^n$ be a projective hypersurface of positive degree defined
over a perfect field $k$ of prime characteristic $p$. We define the level of $H$ as
the level of any element in the below set:
\[
\{f\in K[x_0,\ldots ,x_n] :\ f\text{ is homogeneous, } \mathrm{deg} (f)=h,\  V(f)=H\}/\sim,
\]
where $\sim$ denotes $\operatorname{GL}_{n+1} (k)$--equivalence.
\end{df}

\subsection{Elliptic curves}
From now on, let $k=\mathbb{F}_p$ ($p>2$)and $R=k[x,y,z]$ from now on. Let $E\subseteq\mathbb{P}_{k}^2$ be an elliptic curve
defined over $k$ by a homogeneous cubic $f\in R$. For a prime $l\geq 2$, $E[l](\overline{k})$ is the kernel of the multiplication-by-$l$ isogeny $[l]$, which is a subgroup of $E(\overline{k})$. As mentioned in the introduction, for $l\neq p$, it is $|E[l](\overline{k})|=p^2$ whilst $|E[p](\overline{k})|=p^i$ (for $i=0,1$). If $|E[p](\overline{k})|=p$, then $E$ is said to be ordinary and otherwise supersingular.

Writing  $f^{p-1}=h\cdot (xyz)^{p-1}+\ldots,$ for some $h\in k$, it is also well known that $E$ is supersingular if and only if $h=0$ \cite[V.4.1]{Silverman1986}. The following characterization is given in \cite[Theorem 6.9]{BoixDeStefaniVanzo2015}.

\begin{teo}\label{starting point elliptic setting}
For $E$, $f$ and $p\geq 2$ as above, the following assertions hold.
\begin{enumerate}[(i)]

\item $E$ is ordinary if and only if $f$ has level one.

\item $E$ is supersingular if and only if $f$ has level two.

\end{enumerate}
\end{teo}

Being ordinary or supersingular is related with the level via the Cartier differential operator
\[
\Delta:=\frac{1}{(p-1)!^3}\cdot\frac{\partial^{p-1}}{\partial z^{p-1}}
\frac{\partial^{p-1}}{\partial y^{p-1}}\frac{\partial^{p-1}}{\partial x^{p-1}}.
\]
If $E$ is ordinary, then $h\neq 0$ and therefore one has that
\[
\left(\frac{1}{h}\cdot\Delta\right)(f^{p-1})=1,
\]
which is to say that $f$ has level one. If $h=0$, then the best one can say is that there is a (not necessarily unique) differential operator $\delta$ of level two with the property that $\delta (f^{p^2-1})=f^{p^2-p}$. This operator can be expressed as a certain $R$-multiple (see \cite[Proposition 6.2]{BoixDeStefaniVanzo2015}
for some examples) of the operator
\[
\frac{1}{(p^2-1)!^3}\cdot\frac{\partial^{p^2-1}}{\partial z^{p^2-1}}
\frac{\partial^{p^2-1}}{\partial y^{p^2-1}}\frac{\partial^{p^2-1}}{\partial x^{p^2-1}}.
\]

\subsection{Hyperelliptic curves} Let $\mathcal{C}$ be a hyperelliptic curve of genus $g\geq 2$ defined over $k$ by an affine equation $f(x,y)=0$, with $f(x,y)=y^2-h(x)$, and $h(x)\in k[x]$ a polynomial with no multiple roots and with  degree $2g+2$ (real hyperelliptic curve) or $2g+1$ (imaginary hyperelliptic curve). Notice that the homogeneized  polynomial $\tilde{f}$ defines a plane curve which is singular at infinity, and it is only birationally equivalent to the hyperelliptic curve $\mathcal{C}$, which is smooth in $\mathbb{P}^3(k)$. Moreover, every hyperelliptic curve of genus $g$ is birationally equivalent to an imaginary model via the transformation
\begin{equation}
(x,y)\mapsto \left(\frac{1}{x-a},\frac{y}{(x-a)^{g+1}}\right),
\label{moebiusreal}
\end{equation}
where $a$ is a root of $h$.

As we see in the next example, in general the level is not invariant under this transformation.

\begin{ex}Consider the hyperelliptic curve with imaginary model given by the homogeneous polynomial $f(x,y,z)=y^2z^3-x^5-2x^3z^2-2x^2z^3-xz^4-2z^5$ over $\mathbb{F}_{13}$. The level of this polynomial is $3$. By applying ( \ref{moebiusreal} ) with $a=-1$, we obtain the homogeneous polynomial $h(x,y,z)=y^2z^4-2x^6+2x^4z^2-8x^3z^3+x^2z^4-6xz^5-8z^6$, which has level $2$.
\end{ex}

\begin{df}The level of a hyperelliptic curve of genus $g$ is the level of the polynomial defining its imaginary model, up to change of homogeneous coordinates.
\end{df}

Denote by $\mathrm{Jac}(\mathcal{C})$ the Jacobian of $\mathcal{C}$, which is an abelian variety defined over $k$ of dimension $g$.

The $p$-rank of $\mathcal{C}$ is defined to be the $p$-rank of $\mathrm{Jac}(\mathcal{C})$. Likewise, we say that $\mathcal{C}$ is ordinary (respectively, supersingular, superspecial) if so is $\mathrm{Jac}(\mathcal{C})$.

The ordinary/supersingular/superspecial character of $\mathcal{C}$ can be easily described by the Cartier operator, which we recall next.

\begin{df}\label{the cartier manin matrix in general}
For a positive integer $k\geq 1$, write
\[
h(x,1)^{(p^k-1)/2}=\sum_{j=0}^{N_k} c^{(k)}_j x^j,
\]
where $N_k:=((p^k-1)/2)(2g+1)$. It is convenient to define the $g\times g$ matrix with elements on $k$ 
\[
C_k:=\begin{pmatrix} c^{(k)}_{p^k-1} & c^{(k)}_{p^k-2} & \ldots & c^{(k)}_{p^k-g}\\
c^{(k)}_{2p^k-1}& c^{(k)}_{2p^k-2}& \ldots & c^{(k)}_{2p^k-g}\\
\vdots & \vdots & \ddots & \vdots\\
c^{(k)}_{gp^k-1}& c^{(k)}_{gp^k-2}& \ldots & c^{(k)}_{gp^k-g}\end{pmatrix}.
\]
Set $c_j:=c^{1}_j$ and notice that $C:=C_1$ is the Cartier-Manin matrix of $\mathcal{C}$. It corresponds to the Cartier operator acting on $H^0(\mathcal{C},\Omega_{\mathcal{C}})$, where $\Omega_{\mathcal{C}}$ stands for the sheaf of regular differentials on $\mathcal{C}$. Notice that since our curve is defined over $k$, the Cartier operator is linear (cf. \cite{achter} for details and a relevant recent discussion on the Cartier-Manin and Hasse-Witt matrices). 
\label{numerosc}
\end{df}

We will also make use of the following property.

\begin{lm}\label{high hassewitt} For $k\ge 1$, it is
$$
C_k=C^k.
$$
\end{lm}	
\begin{proof}
The case $k=1$ is trivial. Since 
$$
h(x,1)^{(p^k-1)/2}=(f(x,1)^{(p^{k-1}-1)/2})^ph(x,1)^{(p-1)/2},$$ 
it is 	
$$
C_k=C_{k-1}^pC,
$$
and since the coefficients of $f$ belong to $k$, they are invariant under the $p$-Frobenius and, by induction, the result holds.
\end{proof}

The matrices $C_k$ are related with the ordinary/supersingular/superspecial character of $\mathcal{C}$ in the following way.

\begin{teo}\cite[Theorem 3.1]{Yui1978}
Let $\mathcal{C}$ be a hyperelliptic curve of genus $g\geq 2$ defined over $k$  where $p\geq 7$. Then  $\mathcal{C}$ is ordinary if and only if $\rank (C)=g.$
\label{yuith}
\label{conditions and the Cartier Manin matrix}
\end{teo}

In the other extreme, Nygaard proved \cite[Theorems 4.1]{Nygaard1981} that $\mathcal{C}$ is superspecial if and only if $C$ is identically zero.

As for the supersingular case, the following result holds.

\begin{teo}\cite[Theorem 2.1]{Nygaard1983}
Let $\mathcal{C}$ be a hyperelliptic curve of genus $g\geq 2$ defined over $k$. Then, $\mathcal{C}$ is supersingular if and only if, for $1\leq i,j\leq g$ , it is

\[
c^{(g+2(n-1))}_{ip^{g+2(n-1)-j},g+(n-1)}\equiv 0\pmod{p^n},
\]

where $1\leq n\leq G$, with
\[
G=\begin{cases}
\frac{g^2+1}{2}-(g-1),\text{ if }g\text{ is odd ,}\\
\frac{g^2+1}{2}-\frac{3(g-1)}{2}\text{ if }g\text{ is even}.\end{cases}
\]
\label{conditions and the Cartier Manin matrix2}
\end{teo}

Notice that for $g=2$, Lemma \ref{high hassewitt} and Theorem \ref{conditions and the Cartier Manin matrix2} imply that $\mathcal{C}$ is supersingular if and only iff $C^2$ is identically zero. Moreover, this is equivalent to the $p$-rank being equally $0$.

\section{Results on the level of hyperelliptic curves}\label{results on the level}

Hereafter, unless otherwise is specified, assume  $p\geq 3$ and let $f(x,y,z)=y^2z^{2g+1}-h(x,z)$ with $h(x,z)\in k[x,z]$ a homogeneous polynomial of degree $2g+1$. For any field extension $K/k$, integers $r,n\geq 1$ and any matrix $G\in M_{r\times r}(K)$, denote by $G^{[n]}$ the matrix obtained by raising each entry of $G$ to the $n$-th power. Since the coefficients of $f$ belong to $k$, so are the numbers $c_i$ in Definition \ref{numerosc}, and hence, for the Cartier-Manin matrix $C$ of the curve defined by $f$, it is $C^{[p^r]}=C$ for each $r\geq 1$.

Unlike the elliptic case, the following holds:

\begin{prop}\label{level not one for higher genus: first try}Let $\mathcal{C}$ be a hyperelliptic curve over $k$ of genus $g\geq 2$ Then, if $p\geq 2g-1$, the level of $\mathcal{C}$ is strictly bigger than $1$.
\end{prop}
\begin{proof}
Let $\mathcal{C}$ be defined by a homogeneous polynomial $f\in R$ of degree $d=2g+1$. It is enough to check that the level of $f$ is strictly bigger than $1$. Indeed, if it were not the case, we would have $R=I_1(f^{p-1})$ and so, $1\in I_1(f^{p-1})$. But if we write $f^{p-1}=\displaystyle \sum_{||\alpha||\leq p-1} f_{\alpha}^px_{\alpha}$, for $\{x_{\alpha}\}_{||\alpha||\leq 3(p-1)}$ a basis of $R$ as $R^p$-module and since all the terms $f_{\alpha}^px_{\alpha}$ are distinct, it follows that 
$$p\mathrm{deg}(f_{\alpha})+\mathrm{deg}(x_{\alpha})=(2g+1)(p-1),$$
hence $\mathrm{deg}(f_{\alpha})\geq (2g-2)(1-1/p)$, so that $\mathrm{deg}(f_{\alpha})\geq 2g-2$ if $p\geq 2g-1$. Since, as seen in  \ref{remarkI}, $I_1(f^{p-1})$ is generated by the elements $f_{\alpha}$, this contradicts that $1\in I_1(f^{p-1})$.
\end{proof}

\subsection{Preliminary calculations}
Our main result will rely on the computations which we provide here. We will denote for now on $f(x,y,z)=y^2z^{2g-1}-h(x,z)$ with $h(x,z)\in k[x,z]$ a homogeneous polynomial of degree $2g+1$. For $r\geq 1$ write 
\begin{equation}\label{binomial expansion of curve}
f^{p^r-1}=\sum_{j=0}^{p^r-1}\binom{p^r-1}{j}y^{2(p^r-1-j)}z^{(2g-1)(p^r-1-j)}h(x,z)^j,
\end{equation}
and
\begin{equation}\label{Rp expansion of curve}
f^{p^r-1}=\sum_{||\alpha||\leq p^r-1}f_{\alpha}^{p^r}x^{\alpha},
\end{equation}
where the elements $\{x^{\alpha}\}_{||\alpha||\leq p^r-1}$ are an $R^{p^r}$-basis of $R$.

For a polynomial $h(x,y,z)\in R$, $\mathrm{deg}_x(h)$ stands for the degree of $h$ as a polynomial in $x$ with coefficients in the ring $k[y,z]$ and analogously for $\mathrm{deg}_y(h)$ and $\mathrm{deg}_z(h)$.

\begin{prop}Let $g\geq 2$ and $r\ge 1$ be an integer. Then the following assertions hold.
\begin{enumerate}
	\item[(i)] An $R$-multiple of $y^2$ cannot be an element of a generating set for $I_r$,
    \item[(ii)] Assume $p^r>2g^2-1$. Then, no monomial of the form $x^ayz^b$ can be an element of the generating set of $I_r$ if $a+b<2g-2$.
    \item[(iii)] Assume $p^r>2g^2-1$. Then, the elements $\{z^ix^{2g-2-i}\}_{i=0}^{g-2}$ cannot be generators of $I_r$.
    
    \end{enumerate}
   \label{lemaideal1}
\end{prop}	
\begin{proof}Since the degree in $y$ of each monomial in $f^{p^r-1}$ is between $0$ and $2(p^r-1)$, which is smaller than
  $2p^r$, assertion (i) holds.
  
For assertion (ii), assume that $x^ayz^b=f_{\alpha}$ for some $||\alpha||\leq p^r-1$ and $a+b<2g-2$. since $\mathrm{deg}_z(f_{\alpha}x^{\alpha})=bp^r+\mathrm{deg}_z(x^{\alpha})$, it must be $\mathrm{deg}_z(f_{\alpha}x^{\alpha})<(2g-2)p^r$. Hence, $a+b=2g-3$, for if it were smaller, then
$$
\mathrm{deg}(x^{\alpha})=(2g+1)(p^r-1)-(a+b+1)p^r>2p^r-2g+1>p^r,
$$
which is a contradiction. 

Write $\mathrm{deg}(f_{\alpha}x^{\alpha})=(2g-2)p^r+(2g+1)(p^r-1)-(2g-2)p^r=(2g-2)p^r+3p^r-(2g+1)$. This forces $\mathrm{deg}(x^{\alpha})=3p^r-(2g+1)$.

Setting $x^{\alpha}=x^{p^r-a_1}y^{p^r-a_2}z^{p^r-a_3}$ so that $\mathrm{deg}_y(f_{\alpha}^{p^r}x^{\alpha})=p^r-a_2$, it is $2g+1\geq 2p^r-\mathrm{deg}_y(f_{\alpha}^{p^r}x^{\alpha})$. On the other hand, from (\ref{binomial expansion of curve}) we see that 
$$
\mathrm{deg}_z(f_{\alpha}^{p^r}x^{\alpha})\geq \frac{2g-1}{2}\mathrm{deg}_y(f_{\alpha}^{p^r}x^{\alpha}),
$$
hence
$$
2g+1\geq 2p^r-\mathrm{deg}_y(f_{\alpha}^{p^r}x^{\alpha})\geq 2p^r-\frac{2}{2g-1}\mathrm{deg}_z(f_{\alpha}^{p^r}x^{\alpha})>\frac{2p^r}{2g-1}.
$$
This would yield $p^r\leq 2g^2-1$, which is a contradiction.

For (iii), if any monomial in the list were a generator $f_{\alpha}$, then $\mathrm{deg}_y(f_{\alpha}^{p^r}x^{\alpha})=p^r-a_2$ with $a_2\leq 2g+1$, as in (ii). Then, it would be
$$
\mathrm{deg}_z(f_{\alpha}^{p^r}x^{\alpha})\geq\frac{2g-1}{2}\mathrm{deg}_y(f_{\alpha}^{p^r}x^{\alpha})\geq\frac{(2g-1)(p^r-2g-1)}{2}=(g-1)p^r+\frac{p^r-4g^2+1}{2},
$$
a contradiction with the fact that the $z$-degrees of monomials are less than or equal $g+1$.
\end{proof}	

\begin{prop}Assume $p^r>2g^2-1$. Then the ideal $I_r$ is contained in the ideal
$$
M=(\{z^ax^b|a+b=2g-2,a\geq g-1,b\geq 0\}\cup\{z^ax^b|a+b=2g-1,0\leq a<g-1\}).
$$ 
\end{prop}\label{lemaideal2}
\begin{proof}First, we observe that for $p^r>2g+1$, it must be $2g-2 \leq \mathrm{deg}(f_{\alpha})\leq 2g$. From Proposition \ref{lemaideal1} (i), the generators $f_{\alpha}$ cannot be multiple of $y^2$, so they belong to two categories: a) monomials of the form $x^ayz^b$ with $a+b\geq 2g-2$ (due to Proposition \ref{lemaideal1} (ii)) or b) monomials of the form $x^az^b$ with $a+b\leq 2g$.

For the elements of the form $x^az^b$, it must be $a+b\geq 2g-2$, so we have two cases:

\begin{itemize}
\item[a)] $a+b=2g-2$. This provides the list $\{x^iz^{2g-2-i}\}_{i=0}^{g-1}$, because of Proposition \ref{lemaideal1} (iii).
\item[b)] $a+b\geq 2g-1$. For $a+b=2g-1$, the monomials $z^ax^{2g-1-a}$ with $g-1\leq a\leq 2g-1$ are multiples of the elements in the list of case a), so this case provides the extra terms $\{z^ix^{2g-1-i}\}_{i=0}^{g-2}$. The monomials with $a+b=2g$ are multiples of cases a) or b).
\end{itemize}

Finally, the elements $x^ayz^b$ with $a+b\geq 2g-2$ are multiples of the elements without $y$.
\end{proof}

\begin{df}The ideal $M=(\{z^ax^b|a+b=2g-2,a\geq g-1,b\geq 0\}\cup\{z^ax^b|a+b=2g-1,0\leq a<g-1\})$ will be called the relevant ideal of $\mathcal{C}$.
\end{df}

\subsection{Ordinary hyperelliptic curves of genus $g\geq 2$}
We use here our former results to prove that the level of hyperelliptic curves of genus $g\geq 2$ is $2$.

\begin{teo}\label{main result in the ordinary case}Let $\mathcal{C}$ be a hyperelliptic curve over $k$ defined by a homogeneous polynomial of degree $2g+1$. Assume $p>2g^2-1$. If $\mathcal{C}$ is ordinary then 
$$I_1=I_2=M.$$ 
Therefore the level of the $\mathcal{C}$ is $2$.
\end{teo}				
\begin{proof}
First, for $r=1,2$, notice that for each $1\leq j\leq g$, the expansion of $f^{p^r-1}$ contains the linear combination
$$
\lambda(c_{p^r-j}z^{2g-2}+c_{2p^r-j}z^{2g-3}x+...+c_{gp^r-j}z^{g-1}x^{g-1})^{p^r})x^{p^r-j}y^{p^r-1}z^{p^r-(2g-j)},
$$
with $\lambda=(-1)^{\frac{p^r-1}{2}}{{p^r-1}\choose{\frac{p^r-1}{2}}}\neq 0$, hence for $1\leq j\leq g$, 
$$c_{p^r-j}z^{2g-2}+c_{2p^r-j}z^{2g-3}x+...+c_{gp^r-j}z^{g-1}x^{g-1}\in I_r.$$

Since both $\mathcal{C}$ and $\mathcal{C}^2$ are invertible (because of Theorem \ref{yuith}), it follows that the first $g$ generators of $M$, namely $\{z^ax^b\}_{a=g-1}^{2g-2}$, belong to $I_r$.

We need now to check that the the rest of the generators of $M$, i.e. $\{z^ax^b\}_{a=0}^{g-1}$ also belong to $I_r$. Let us start showing that $x^{g-1}z^{g-2}\in I_r$.

Set $u:=\left\lfloor\frac{(g-1)p^r}{2g-1}\right\rfloor-1$ and consider
$$
h_{u,1}(x,y,z)={{p^r-1}\choose{u}}(y^2z^{2g-1})^u(-x^{2g+1})^{p^r-1-u}.
$$
it is
$$
(g-1)p^r-2(2g-1)\leq \mathrm{deg}_z(h_{u,1})\leq (g-1)p^r-(2g-1),
$$
and
$$
(g+1)p^r+\frac{p^r}{2g-1}\leq \mathrm{deg}_x(h_{u,1})\leq (g+1)p^r-(2g+1).
$$

Since $p^r>2g^2-1$, it follows that $(g-2)p^r\leq \mathrm{deg}_z(h_{u,1})\leq {(g-2)p^r+(p^r-1)}$ and $(g+1)p^r\leq \mathrm{deg}_x(h_{u,1})\leq (g+1)p^r+(p^r-1)$, so the term $x^{g+1}z^{g-2}$ is a summand in some of the elements $f_{\alpha}$, generators of $I_r$.

Since $u<p^r-1$, $h_{u,1}$ only appears in the $u$-th term of the binomial sum (\ref{binomial expansion of curve}):
\begin{equation}
(y^2z^{2g-1})^u(g(x,z))^{p^r-1-u}=\mu h_{u,1}(x,y,z)+y^{2u}z^{(2g-1)u}[x^{(2g+1)(p^2-2-u)}+...]
\label{eqnaux}
\end{equation}
with $\mu\neq 0$. Setting $h_{u,1}(x,y,z)=x^{g+1}z^{g-2}x^{}\alpha$, with $x^{\alpha}$ an element in the basis of $R$ as $R^{p^r}$-module, and observing in Equation \ref{eqnaux} that the degree in $z$ of right hand summand is  bigger than or equal to $g-1$, the polynomial $f_{\alpha}$ containing the monomial $x^{g+1}z^{g-2}$ consists in the sum of a) this monomial (homogeneous of degree $2g-1$, b) a homogeneous component of degree $2g-2$ with degree in $z$ bigger than or equal to $g-1$ and no $y$ and c) a sum of homogeneous polynomials of smaller degrees.

Since $I_r$ is a homogeneous ideal, it follows that $f_{\alpha}$ contains, in its expansion, the term
$$
\rho x^{g+1}z^{g-2}+\sum_{i=g-1}^{2g-2}a_iz^ix^{2g-2-i}\in I_r$$ 
with $\rho\neq 0$. And since the terms $z^ix^{2g-2-i}$ belong to $I_r$ for $g-1\leq i\leq 2g-2$, also $x^{g+1}z^{g-2}\in I_r$.

The rest of the elements $x^{g+k}z^{g-k-1}$ can be now inductively proved to belong to $I_r$ for $2\leq k\leq g-1$: set
$$
u_k:=\left\lfloor\frac{(g-k)p^r}{2g-1}\right\rfloor-1
$$
and
$$
h_{u_k,k}(x,y,z)={{p^r-1}\choose{u_k}}(y^2z^{2g-1})^u(-x^{2g+1})^{p^r-1-u_k}.
$$
With the same argument as for $x^{g+1}z^{g-2}$, we show that $x^{g+k}z^{g-1-k}$ appears in the expansion of some generator $f_{\alpha_k}$ of $I_r$ and the the previous homogeneity argument, we conclude that $f_{\alpha_k}$ contains, in its expansion, the term
$$
\rho_k x^{g+k}z^{g-1-k}+\sum_{i=1}^{k-1}\rho_i x^{g+i}z^{g-1-i}+\sum_{i=g-1}^{2g-2}a_iz^ix^{2g-2-i}\in I_r
$$
with $\rho_i\neq 0$, hence $x^{g+k}z^{g-1-k}\in I_r$.

Finally, to finish the proof, we have to see that $x^{2g-1}\in I_r$, take a non-negative integer $0\leq u_r<2g+1$ with $2gp^r\equiv u_r \pmod{2g+1}$, and consider, in the binomial expansion (\ref{binomial expansion of curve}), the term
$$
(y^2x^{2g-1})^{\frac{p^r+u_r-(2g+1)}{2g+1}}(-x^{2g+1})^{\frac{2gp^r-u_r}{2g+1}},
$$
which appears with non-zero coefficient and can be written as
$$
\pm (x^{2g-1})^{p^r}x^{p^r-u_r}y^{\frac{2(p^r+u_r-(2g+1))}{2g+1}}z^{ \frac{(2g-1)(p^r+u_r-(2g+1))}{2g+1} }.
$$
\end{proof}

\subsection{Some partial results in the non-ordinary case}

\begin{lm}\label{ideal} For $r\ge 1$, there are elements generators $h_{r,1},...,h_{r,g}$ of $I_r$ of the form
$$\left[\begin{matrix}
	h_{r,1} & h_{r,2} & \cdots & h_{r,g}
	\end{matrix}\right]=\left[\begin{matrix}
	z^{2g-2} & xz^{2g-3} & \cdots & x^{g-1}z^{g-1}
	\end{matrix}\right]C_r.
$$
\end{lm}	
\begin{proof}
	It is enough to check that the expansion of $(y^2z^{2g-1})^{(p^r-1)/2}g(x,z)^{(p^r-1)/2}$ contains a sum of the form
	$$
	\sum_{j=1}^g\sum_{i=1}^gc^{(r)}_{ip^r-j}z^{(2g-1-i)p^r}x^{(i-1)p^r}x_{\alpha_{i,j}},
	$$
for $x_{\alpha_{i,j}}$ elements in the basis of $R$ as $R^{p^r}$-module. Notice that since the degree of $y$ is between $0$ and $2(p^r-1)$, the power $y^{p^r-1}$ do not appear anywhere else in the expansion of $f(x,y,z)^{p^r-1}$ than in the considered term so than the sum cannot cancel with expansions corresponding to other powers.

But, this is clear:
$$
y^{p^r-1}z^{\frac{(2g-1)(p^r-1)}{2}}c^{(r)}_{ip^r-j}x^{ip^r-j}z^{\frac{(2g+1)(p^r-1)}{2}-ip+j}=c^{(r)}_{ip^r-j}g_{r,i}^{p^r}x_{\alpha_{i,j}},
$$
with $x_{\alpha_{i,j}}=x^{p^r-j}y^{p^r-1}z^{p^r-(2g-j)}$.
\end{proof}

\begin{prop}
For $r\geq 1$, the $k$-linear combinations of the set $\{z^{2g-2-z}x^i \: | \: 0\le i \le g-1\}$ which belong to $I_1$ are generated by  the entries of the following row-vector: 
 $$[\begin{matrix}z^{2g-2} & \cdots & z^{g-1}x^{g-1} \end{matrix}]\cdot [C_g^r \: | \: C_g^{r-1} H_g],$$ 
where $H_g\in M_{g\times s_g}(k)$ with $s_g\geq 1$.
\end{prop}	
\begin{proof}
For $r=1$, write
$$
f^{p-1}=...+\sum_{i=0}^{g-1}a_i(x^iz^{2g-2-i})^p\sum_{\alpha,\beta,\gamma}x^{p-\alpha}y^{p-\beta}z^{p-\gamma}+...
$$
with $\alpha+\beta+\gamma=2g+1$. The value $\beta=1$ yields the product
$$
[\begin{matrix}z^{2g-2} & \cdots & z^{g-1}x^{g-1} \end{matrix}]\cdot [C_g],
$$
while odd values $3\leq \beta\leq 2g+1$ yield the columns
$$
[\begin{matrix}z^{2g-2} & \cdots & z^{g-1}x^{g-1} \end{matrix}]\cdot [H_g].
$$

For $r>1$, the $R^{p^r}$-basis terms in the expansion $f^{p^r-1}$ can be expressed as
$$
x^{p^r-\alpha} y^{p^r-\beta} z^{p^r-\gamma}=(x^{p^{r-1}-a}y^{p^{r-1}-1}z^{p^{r-1}-b})^px^{ap-\alpha} y^{p-\beta} z^{bp-\gamma},
$$
where $a+1+b=2g+1$. Again, collecting the $R^{p}$-coefficients for $\beta=1$ yields
$$
[\begin{matrix}z^{2g-2} & \cdots & z^{g-1}x^{g-1} \end{matrix}]\cdot C_g^r,
$$
and the odd terms $3\leq \beta\leq 2g+1$ yield
$$
[\begin{matrix}z^{2g-2} & \cdots & z^{g-1}x^{g-1} \end{matrix}]\cdot C_g^{r-1}H_g.
$$
\end{proof}

\begin{rk}By choosing a suitable ordering of the columns in the matrix $H_g$, we can assume that its first column is

$$
\left[\begin{array}{c}
c_{p-(g+1)}\\
c_{2p-(g+1)}\\
\vdots\\
c_{gp-(p+1)}
\end{array}
\right],
$$
i.e., it corresponds to the sum $\sum_{i=1}^{g}c_{ij-(g+1)}x^{(2g-1-i)p}z^{(i-1)p}x^{p-(g+1)}y^{p-1}z^{p-(2g-j-1)}$ in the expansion of $g(x,z)^{\frac{p-1}{2}}$.
\label{remH}
\end{rk}

\begin{teo}\label{main result in the supersingular case}
	If $C_g^r\ne 0$ but $C_g^{r+1}=0$, then the level is strictly greater than $r+1$.
\end{teo}	
\begin{proof}Since $C_g^r\ne 0$, there exits a non-zero linear combination of $\{z^{2g-2-z}x^i \: | \: 0\le i \le g-1\}$ in $I_r$. Since $C_g^{r+1}=C_g^{r+2}=0$, there is no non-zero linear combination of $\{z^{2g-2-z}x^i \: | \: 0\le i \le g-1\}$ in $I_{r+2}$. Therefore, $I_r\ne I_{r+2}$ and hence $I_r\ne I_{r+1}$ and $e>r+1$. 
\end{proof}

\begin{cor}Let $\mathcal{C}$ be a supersingular but not superspecial hyperelliptic curve of genus $g\geq 2$. Then, the level is strictly greater than $2$.
\end{cor}
\begin{proof}Since $\mathcal{C}$ is not superspecial, by Theorem 4.1 in \cite{Nygaard1981}, it is $C_g\neq 0$. Since $\mathcal{C}$ is supersingular, by Theorem \ref{conditions and the Cartier Manin matrix2}, it is also $C_g^g=0$, hence there exists $1\geq r<g$ such that $C_g^r\neq 0$ but $C_g^{r+1}=0$ and hence $e>r+1\geq 2$.
\end{proof}

\section{Some examples}\label{some examples}
So far, we have studied the ordinary and supersingular (but not superspecial) cases. In this last section we give computational and theoretical evidence which leads us to pose the following conjecture.

\begin{con}Let $\mathcal{C}$ be a superspecial hyperelliptic curve of genus $g$ over $k$. Then, the level of $\mathcal{C}$ is strictly bigger than $2$.
\end{con}

All the computer calculations appearing here were done with Macaulay2 \cite{M2} and Magma \cite{Magma}.

\subsection{The non-ordinary and non-superspecial cases: some examples}\label{non ordinary examples}

First, we show by three numerical examples that the level cannot be used in general to characterize the $p$-rank of hyperelliptic curves.

\begin{ex}\label{intermediate with also level 2}
Consider the curve defined by the quintic $f=y^2z^3-x^5-2z^5$, for 
$p=11$, $13$ and $17$. 

For $p=11$, the $p$-rank is $2$ (i.e. the curve is ordinary), and as predicted by Theorem 3.5, the algebra package gives level $2$, the chain of ideals of $p^e$th roots being $R \supset (z^2,xz,x^3).$

For $p=13$, the level of $f$ is $4$ and the $p$-rank is $0$. Finally, for $p=17$, the level of $f$ is $3$ and the $p$-rank is also $0$.
\end{ex}

\subsection{The superspecial case: some examples and a conjecture}
The authors do not know if, in general, the level of a superspecial curve is always bigger than or equal $3$. Indeed, denote
$$
\tilde{C_g}=[C_g| H_g].
$$
In the superspecial case, $C_g=0$, but if $\tilde{C_g}\neq 0$, it would be still true that there is a non-zero linear combination of the terms $\{z^{2g-1-i}z^i\}_{i=1}^{g}$ in $I_1$, while, again, there is none in $I_2$, which would yield $e>2$. It is not clear for the moment how to control the numbers $c_i$ from the coefficients of $f$ and the multinomial coefficients, so it is not clear whether in general $\tilde{C_g}=[C_g| H_g]$ vanishes or not. Our numerical examples seem to show that this is the case.

However, this is so for the following two infinite families.

\begin{ex}\label{first superspecial family}
Let $\mathcal{C}_p$ be the curve defined by $f=y^2z^{2g-1}-x^{2g+1}-\mu xz^{2g}$ with non zero $\mu\in k$. By Theorem 2 in \cite{Val1995}, $\mathcal{C}_p$ is superspecial if an only if $p\equiv 2g+1\pmod{4g}$ or $p\equiv -1\pmod{4g}$. In the first case, $c_{p-g-1}\neq 0$, and in the second case, $c_{gp-g-1}\neq 0$, so, by Remark \ref{remH}, it is $\tilde{C_g}\neq 0$ and hence $e>2$.
\end{ex}

Another family:

\begin{ex}\label{second superspecial family}
Let $\mathcal{C}_p$ be the curve defined by $f=y^2z^{2g-1}-x^{2g+1}-\mu z^{2g+1}$ with non zero $\mu\in k$. Again, by Theorem 2 in \cite{Val1995}, $\mathcal{C}_p$ is superspecial if an only if $p\equiv 2g+1\pmod{2g+1}$, what implies that $c_{gp-g-1}\neq 0$ and again $e>2$.
\end{ex}

Computational evidence together with these examples lead us to formulate the following conjecture, with which we conclude our work.

\begin{con}Let $\mathcal{C}$ be a superspecial hyperelliptic curve of genus $g$ over $k$. Then, the level is strictly bigger than $2$.
\end{con}

\bibliographystyle{alpha}
\bibliography{AFBoixReferences}

\newcommand{\etalchar}[1]{$^{#1}$}
\begin{thebibliography}{AMBL05}

\bibitem[AH]{achter}
J.D Achter and E.W. Howe.
\newblock Hasse-{W}itt and {C}artier-{M}anin matrices: a warning and a request.
\newblock Available at \url{https://arxiv.org/abs/1710.10726}.

\bibitem[AMBL05]{AlvarezBlickleLyubeznik2005}
J.~Alvarez-Montaner, M.~Blickle, and G.~Lyubeznik.
\newblock Generators of {$D$}-modules in positive characteristic.
\newblock {\em Math. Res. Lett.}, 12(4):459--473, 2005.

\bibitem[BCP97]{Magma}
W.~Bosma, J.~Cannon, and C.~Playoust.
\newblock The {M}agma algebra system. {I}. {T}he user language.
\newblock {\em J. Symbolic Comput.}, 24(3-4):235--265, 1997.
\newblock Computational algebra and number theory (London, 1993).

\bibitem[BDSV15]{BoixDeStefaniVanzo2015}
A.~F. Boix, A.~De~Stefani, and D.~Vanzo.
\newblock An algorithm for constructing certain differential operators in
  positive characteristic.
\newblock {\em Matematiche (Catania)}, 70(1):239--271, 2015.

\bibitem[BMS08]{BlickleMustataSmith2008}
M.~Blickle, M.~Musta{\c{t}}{\u{a}}, and K.~E. Smith.
\newblock Discreteness and rationality of {$F$}-thresholds.
\newblock {\em Michigan Math. J.}, 57:43--61, 2008.

\bibitem[Gro67]{EGAIV}
A.~Grothendieck.
\newblock \'{E}l\'ements de g\'eom\'etrie alg\'ebrique. {IV}. \'{E}tude locale
  des sch\'emas et des morphismes de sch\'emas {IV}.
\newblock {\em Inst. Hautes \'Etudes Sci. Publ. Math.}, (32):361, 1967.

\bibitem[GS]{M2}
Daniel~R. Grayson and Michael~E. Stillman.
\newblock Macaulay2, a software system for research in algebraic geometry.
\newblock Available at \url{http://www.math.uiuc.edu/Macaulay2/}.

\bibitem[ILL{\etalchar{+}}07]{Twentyfourhours}
S.~B. Iyengar, G.~J. Leuschke, A.~Leykin, C.~Miller, E.~Miller, A.~K. Singh,
  and U.~Walther.
\newblock {\em Twenty-four hours of local cohomology}, volume~87 of {\em
  Graduate Studies in Mathematics}.
\newblock American Mathematical Society, Providence, RI, 2007.

\bibitem[Nyg81]{Nygaard1981}
N.~O. Nygaard.
\newblock Slopes of powers of frobenius on crystalline cohomology.
\newblock {\em Ann. Sci. \'Ecole Norm. Sup.}, 14:369--401, 1981.

\bibitem[Nyg83]{Nygaard1983}
N.~O. Nygaard.
\newblock On supersingular abelian varieties.
\newblock In {\em Algebraic geometry ({A}nn {A}rbor, {M}ich., 1981)}, volume
  1008 of {\em Lecture Notes in Math.}, pages 83--101. Springer, Berlin, 1983.

\bibitem[Sil86]{Silverman1986}
J.~H. Silverman.
\newblock {\em The Arithmetic of Elliptic Curves}, volume 106 of {\em Graduate
  Texts in Mathematics}.
\newblock Springer, second edition, 1986.

\bibitem[Val95]{Val1995}
R.C. Valentini.
\newblock Hyperelliptic curves with zero hasse-witt matrix.
\newblock {\em Manuscripta Math.}, 86(2):185--194, 1995.

\bibitem[Yek92]{Yekutiely1992}
A.~Yekutieli.
\newblock An explicit construction of the {G}rothendieck residue complex.
\newblock {\em Ast\'erisque}, (208):127, 1992.
\newblock With an appendix by Pramathanath Sastry.

\bibitem[Yui78]{Yui1978}
N.~Yui.
\newblock On the {J}acobian varieties of hyperelliptic curves over fields of
  characteristic {$p>2$}.
\newblock {\em J. Algebra}, 52(2):378--410, 1978.

\end{thebibliography}

\end{document}